\documentclass{amsart}

\usepackage{amssymb}
\usepackage{amsthm}

\usepackage{amsfonts}
\usepackage{mathrsfs}
\usepackage{amsmath}

\newtheorem{theorem}{Theorem}
\newtheorem{lemma}[theorem]{Lemma}
\newtheorem{proposition}[theorem]{Proposition}
\newtheorem{thm}{Theorem}
\newtheorem{conj}{Conjecture}

\newenvironment{thmbis}[1]
 {\renewcommand{\thethm}{\ref{#1}$'$}%
 \addtocounter{thm}{-1}%
 \begin{thm}}
 {\end{thm}}

\newenvironment{thmbiss}[1]
 {\renewcommand{\thethm}{\ref{#1}$''$}%
 \addtocounter{thm}{-1}%
 \begin{thm}}
 {\end{thm}}

\numberwithin{equation}{section}

\newcommand{\ball}{\mathbb{B}^n}
\newcommand{\sphere}{\mathbb{S}^{n-1}}
\newcommand{\bfell}{\boldsymbol{\ell}_{\alpha}}
\newcommand{\bbR}{\mathbb{R}}
\newcommand{\bfe}{\boldsymbol{e}}
\newcommand{\bfn}{\boldsymbol{n}}

\title{A proof of the generalized Khavinson conjecture}

\thanks{This work was supported by the National Natural Science
Foundation of China grants 11571333 and 11971453.}

\author{Congwen Liu}
\email{cwliu@ustc.edu.cn}

\address{School of Mathematical Sciences,
University of Science and Technology of China,\\
Hefei, Anhui 230026,
People's Republic of China\\
and\\
Wu Wen-Tsun Key Laboratory of Mathematics\\
USTC, Chinese Academy of Sciences}


\begin{document}

\begin{abstract}
We give a complete proof of the generalized Khavinson conjecture which states that, for bounded harmonic functions
on the unit ball of $\bbR^n$, the sharp constants in the estimates for their radial derivatives and for their gradients
coincide.
\end{abstract}

\maketitle

\section{Introduction}

For a fixed positive integer $n\geq 3$, let $\ball$ be the open unit ball in $\bbR^n$
and $\sphere:=\partial \ball$.
Let $h^{\infty}$ be the space of bounded harmonic functions on $\ball$.
For fixed $x\in \ball$ let $C(x)$ denote the smallest number such that the estimate
\[
\left| \nabla u(x)\right| \leq C(x) \sup_{y\in \ball} \left|u(y)\right|
\]
holds for all $u\in h^{\infty}$. Similarly, for $x\in \ball$ and $\boldsymbol{\ell} \in \mathbb{S}^{n-1}$,
denote by $C(x,\boldsymbol{\ell})$ 
the smallest number such that the inequality
\begin{equation}
\left|\langle \nabla u(x), \boldsymbol{\ell} \rangle\right| \leq C(x,\boldsymbol{\ell}) \sup_{y\in \ball}|u(y)|
\end{equation}
holds for all $u\in h^{\infty}$. As is easily shown (see \cite[Chapter 6]{KM12}), for any $x\in \ball$, both $C(x)$
and $C(x,\boldsymbol{\ell})$ are finite. Also, since
\[
\left| \nabla u(x)\right| = \sup_{\boldsymbol{\ell}\in \sphere} \left| \langle \nabla u(x), \boldsymbol{\ell}\rangle\right|,
\]
we clearly have
\begin{equation}\label{eqn:varproblem}
C(x)= \sup_{\boldsymbol{\ell}\in \sphere} C(x,\boldsymbol{\ell}).
\end{equation}

The generalized Khavinson conjecture states:

\begin{conj}\label{conj}
 For $x\in \ball\setminus \{0\}$ we have
 \[
 C(x)=C(x,\bfn_{x}),
\]
where $\bfn_x:=x/|x|$ is the unit outward normal vector to the sphere $|x|\sphere$ at $x$.
\end{conj}

This conjecture was formulated by G. Kresin and V. Maz'ya in \cite{KM10a}. It actually dates back to 1992.
D. Khavinson \cite{Kha92} obtained a sharp pointwise estimate for the radial derivative of bounded harmonic functions
on the unit ball of $\bbR^3$. In a private conversation with K. Gresin and V. Maz'ya, he conjectured that
the same estimate holds for the norm of the gradient of bounded harmonic functions.
Estimates of such type are of use in problems relating electrostatics as well as hydrodynamics of ideal fluid,
elasticity and hydrodynamics of the viscous incompressible fluid (see, for instance, the books by Protter
and Weinberger \cite{PW84}, G. Kresin and V. Maz'ya \cite{KM12}).

In 2010, G. Kresin and V. Maz'ya \cite{KM10b} proved the half-space analogue of the above conjecture.
However, it turned out that the original conjecture is very difficult.

In 2017, D. Kalaj \cite{Kal17} showed that the conjecture is true for $n=4$. Very recently,
P. Melentijevi\'c \cite{Mel19} confirmed the conjecture in $\bbR^3$.

For $n\geq 5$, only partial results are available. See \cite[Chapter 6]{KM12} for solutions of various
Khavinson-type extremal problems for harmonic functions on the unit ball and on a half-space in $\bbR^n$.
Recently,  M. Markovi\'c \cite{Mar17} proved the conjecture when $x$ is near the boundary of the unit ball.

The aim of this note is to prove the following.

\begin{theorem}\label{thm:main}
The generalized Khavinson conjecture is correct.
\end{theorem}

Just like that in \cite{Kal17}, \cite{Mar17} and \cite{Mel19}, our proof is
based on an observation of M. Markovi\'c in \cite{Mar17} that the generalized Khavinson conjecture is equivalent
to the statement that the optimization problem
\begin{equation}\label{eqn:extremal}
\sup_{\alpha} C(\rho \bfe_1, \bfell)
\end{equation}
has a solution at $\alpha=0$, where
\[
\bfell := \bfe_1 \cos \alpha + \bfe_2 \sin \alpha, \quad \alpha \in [0,\pi],
\]
with $\bfe_1$ and $\bfe_2$ the first two basis vectors in $\bbR^n$.
However, to solve this optimization problem, we find a new representation of $C(\rho \bfe_1,\bfell)$ in terms of the Gegenbauer polynomials
(Proposition \ref{prop:keyreprn}) and reduce the problem to showing that $C(\rho \bfe_1,\bfell)$
is a convex function of $\cos \alpha$ (Theorem \ref{thm:reformlmain}). The key ingredients in the proof are the addition theorem for the Gegenbauer polynomials
(\eqref{eqn:addthm} in Section 2), a variant of Gegenbauer's product formula
(Lemma \ref{lem:Gegenbauer1874}) and the positivity of a certain series involving Gegenbauer polynomials.

Using an explicit formula for $C(x, \bfn_x)=C(|x|\bfe_1, \boldsymbol{\ell}_0)$ (see \eqref{eqn:Melen} below), we can reformulate Theorem \ref{thm:main} as follows.

\begin{theorem}\label{thm:main2}
For every $u\in h^{\infty}$ and every $x\in \ball$, we have the following sharp inequality:
\[
|\nabla u(x)| ~\leq~ \frac{c_n}{1-|x|^2} \left\{\int\limits_{-1}^1
\frac {\left| t - \frac{n-2}n |x| \right| (1-t^2)^{\frac {n-3}{2}}}
{(1-2t|x|+|x|^2)^{\frac {n-2}{2}}} dt \right\}\|u\|_{\infty},
\]
where $c_n:= \frac {2\Gamma(\frac {n+2}{2})} {\Gamma(\frac{1}{2}) \Gamma(\frac {n-1}{2})}$, with $\Gamma(x)$ the Gamma function.
\end{theorem}

\section{Preliminaries on the Gegenbauer polynomials}

The Gegenbauer polynomial $C_k^{\lambda}(x)$ of degree $k$ associated to $\lambda$ is defined to be
the coefficient of $z^k$ in the expansion of $(1-2xz+z^2)^{-\lambda}$ in powers of $z$:
\begin{equation}\label{eqn:generatingformula}
(1-2xz+z^2)^{-\lambda} = \sum_{k=0}^{\infty} C_k^{\lambda}(x) z^k, \qquad -1<x<1,\; |z|<1 .
\end{equation}

We collect here, for the reader¡¯s convenience, all necessary facts on
the Gegenbauer polynomials.

\begin{enumerate}
\item[(i)]
Explicit representation (\cite[p.175, (18)]{EMOT53b}): if $\lambda>-1/2$,
\begin{equation}\label{eqn:explicit}
C_k^{\lambda}(x)=\sum_{j=0}^{[ k/2 ]}
\frac{(-1)^j (\lambda)_{k-j}}{j!\,(k-2j)!}\,(2x)^{k-2j}.
\end{equation}
In particular,
\begin{equation}\label{eqn:specialcases}
C_0^{\lambda}(x)=1, \quad C_1^{\lambda}(x)=2\lambda x.
\end{equation}
Here and throughout the paper, $(\lambda)_k$ denotes the Pochhammer symbol (or the
shifted factorial) which is defined by
\[
(\lambda)_0 := 1, \quad (\lambda)_k := \lambda(\lambda+1)\ldots(\lambda+k-1)
\quad \text{ for } k\geq 1.
\]

\item[(ii)]
Orthogonality relation (\cite[p.177, (16) and (17)]{EMOT53a}): if $\lambda\neq 0$,
\begin{equation}\label{eqn:orthorelation}
\int\limits_{-1}^1 C_k^{\lambda}(x) \, C_l^{\lambda}(x)\, (1-x^2)^{\lambda-\frac{1}{2}} \,d x
~=~ \begin{cases}
0, & k\neq l,\\
\dfrac{\Gamma(\frac {1}{2})\,\Gamma(\lambda+\frac {1}{2})\, (2\lambda)_k}{\Gamma(\lambda) \, (k+\lambda)\, k!},& k=l.
\end{cases}
\end{equation}

\item[(iii)]
Differentiation formula (\cite[p.176, (23)]{EMOT53b}): for $m \leq k$,
\begin{equation} \label{eqn:diffofgeng1}
\frac {d^m}{d x^m} C_k^{\lambda}(x) = 2^m (\lambda)_m \,C_{k-m}^{\lambda+m}(x).
\end{equation}

\item[(iv)]
Rodrigues' formula (\cite[p.175, (11)]{EMOT53b}):
\begin{equation}\label{eqn:Rodrigues}
C_k^{\lambda}(x) = \frac {(-1)^k (2\lambda)_k}{2^k k! (\lambda+\frac {1}{2})_k} (1-x^2)^{\frac {1}{2} -\lambda}
\frac {d^k}{dx^k} \left\{(1-x^2)^{k+\lambda-\frac {1}{2}} \right\}.
\end{equation}

\item[(v)]
Gegenbauer's addition theorem (\cite[p.30, (4.7)]{Ask75}):
\begin{align} \label{eqn:addthm}
C_k^{\lambda}(& \cos \theta \cos \varphi + \sin \theta \sin \varphi \cos \psi)\\
&~=~ \frac {\Gamma(2\lambda -1)} {\Gamma^2(\lambda)} \sum_{j=0}^{k}
\frac {2^{2j} \Gamma(k-j+1) \Gamma^2 (\lambda+j)} {\Gamma (k+2\lambda+j)} (2\lambda+2j-1)
\notag\\
& \qquad \quad \times (\sin\theta)^j (\sin\varphi)^j  C_{k-j}^{\lambda+j} (\cos\theta) C_{k-j}^{\lambda+j} (\cos\varphi)
C_j^{\lambda -\frac {1}{2}}(\cos \psi). \notag
\end{align}

\item[(vi)]
Gegenbauer's product formula (\cite[p.30, (4.10)]{Ask75}):
\begin{align}\label{eqn:prodform}
C_{k}^{\lambda} &(\cos\varphi) C_k^{\lambda}(\cos \psi) \\
& ~=~ \frac {\Gamma(\lambda+\frac {1}{2})} {\Gamma (\frac {1}{2}) \Gamma(\lambda)} \frac {(2\lambda)_k}{k!}
\int\limits_{0}^{\pi}  C_k^{\lambda}( \cos \varphi \cos \psi + \sin \varphi \sin \psi \cos \theta)
(\sin\theta)^{2\lambda-1} d\theta. \notag
\end{align}
\end{enumerate}

For $\lambda> 0$ we write
\[
\widetilde{K}_{\lambda}(x,y,z) ~:=~ \frac {\Gamma(\lambda+\frac {1}{2})}{\Gamma(\lambda)\Gamma(\frac {1}{2})} \frac {(1-x^2-y^2-z^2+2xyz)^{\lambda-1}}
{(1-x^2)^{\lambda-\frac {1}{2}} (1-y^2)^{\lambda-\frac {1}{2}}}, \quad x,y,z\in (-1,1).
\]

The following variant of Gegenbauer's product formula plays a key role in our proof of the generalized Khavinson conjecture.
\begin{lemma}\label{lem:Gegenbauer1874}
If $\lambda>0$ and $-1<x,y<1$, then
\begin{equation}\label{eqn:prodfml}
C_k^{\lambda}(x) C_k^{\lambda}(y) ~=~ \frac {(2\lambda)_k} {k!}
\int\limits_{-1}^1 C_k^{\lambda}(z) K_{\lambda}(x,y,z) dz
\end{equation}
with
\[
K_{\lambda}(x,y,z) ~:=~ \begin{cases}
\widetilde{K}_{\lambda}(x,y,z), & \text{if }\; 1-x^2-y^2-z^2+2xyz>0,\\
0,& \text{otherwise.}
\end{cases}
\]
\end{lemma}

\begin{proof}
This follows immediately from a change of variables in Gegenbauer's product formula \eqref{eqn:prodform}.
It is also a special case of Theorem 1 of \cite{Gas71}.
\end{proof}

\begin{lemma}
If $\lambda\neq 1$ then
\begin{equation}\label{eqn:diffofgeng2}
\frac {d}{dx} \left\{(1-x^2)^{\lambda-\frac {1}{2}} C_k^{\lambda}(x) \right\}
~=~  - \frac {(k+1)(k+2\lambda-1)}{2(\lambda-1)} (1-x^2)^{\lambda-\frac {3}{2}} C_{k+1}^{\lambda-1}(x).
\end{equation}
\end{lemma}

\begin{proof}
This is immediate from Rodrigues' formula \eqref{eqn:Rodrigues}.
\end{proof}

\begin{lemma}\label{lem:intbyterms}
Let $\lambda>-1/2$ and $-1<s<1$. Then we have
\begin{align}
\label{eqn:intbyterms}
\int\limits_{-1}^{1} & |x-s| (1-x^2)^{\lambda-\frac {1}{2}} C_k^{\lambda}(x) dx \\
&~=~ \frac {8\lambda(\lambda+1)} {k(k-1)(k+2\lambda)(k+2\lambda+1)}
(1-s^2)^{\lambda+\frac {3}{2}}  C_{k-2}^{\lambda+2} (s) \notag
\end{align}
for  $k=2,3,\ldots$.
\end{lemma}

\begin{proof}
By using \eqref{eqn:diffofgeng2} and integrating by parts, we obtain
\begin{align*}
\int\limits_{-1}^{s}  (s-x) & (1-x^2)^{\lambda-\frac {1}{2}} C_k^{\lambda}(x) dx\\
& \qquad ~=~ - \frac {2\lambda} {k(k+2\lambda)} \int\limits_{-1}^{s}  (s-x)
d\left((1-x^2)^{\lambda+\frac {1}{2}} C_{k-1}^{\lambda+1}(x)\right)\\
& \qquad ~=~ \frac {2\lambda} {k(k+2\lambda)} \int\limits_{-1}^{s}
(1-x^2)^{\lambda+\frac {1}{2}} C_{k-1}^{\lambda+1}(x) dx\\
& \qquad ~=~ \frac {4\lambda(\lambda+1)} {(k-1)k(k+2\lambda)(k+2\lambda+1)}
(1-s^2)^{\lambda+\frac {3}{2}} C_{k-2}^{\lambda+2}(s),
\end{align*}
where the last equality again follows from \eqref{eqn:diffofgeng2}.
In the same way we get
\begin{align*}
\int\limits_{s}^{1}  (x-s) & (1-x^2)^{\lambda-\frac {1}{2}} C_k^{\lambda}(x) dx\\
& \qquad ~=~ \frac {4\lambda(\lambda+1)} {(k-1)k(k+2\lambda)(k+2\lambda+1)}
(1-s^2)^{\lambda+\frac {3}{2}} C_{k-2}^{\lambda+2}(s).
\end{align*}
Then adding these two identities gives \eqref{eqn:intbyterms}.
\end{proof}

\section{A new representation formula for $C(\rho \bfe_1,\bfell)$}

The following representation of $C(\rho \bfe_1,\bfell)$ in terms of the Gegenbauer polynomials is very
efficient for solving the extremal problem \eqref{eqn:extremal}.

\begin{proposition}\label{prop:keyreprn}
We have
\begin{align*} 
C(\rho \bfe_1, \bfell) ~=~&  \frac {c_n} {1-\rho^2}
\Biggl\{\int\limits_{-1}^1 \left| \frac {n-2}{n} \rho \cos\alpha -x \right| (1-x^2)^{\frac {n-3}{2}} dx \\
& \quad + (n-2)\rho \cos \alpha \int\limits_{-1}^1 \left| \frac {n-2}{n} \rho \cos\alpha -x \right| (1-x^2)^{\frac {n-3}{2}} x dx \notag\\
& \quad + \frac {2}{n^2-1} \left[1-\frac {(n-2)^2}{n^2} \rho^2 \cos^2 \alpha\right]^{\frac {n+1}{2}} \notag\\
& \qquad \quad \times \sum_{k=2}^{\infty} \frac {(k-2)!}{(n+2)_{k-2}} C_{k-2}^{\frac {n+2}{2}} \left(\frac {n-2}{n} \rho \cos \alpha\right)
C_k^{\frac {n-2}{2}}(\cos \alpha) \rho^k \Biggr\} \notag
\end{align*}
for any $\rho\in [0,1)$ and  any $\alpha \in [0,\pi]$, where $c_n:= \frac {2\Gamma(\frac {n+2}{2})} {\Gamma(\frac{1}{2}) \Gamma(\frac {n-1}{2})}$.
\end{proposition}


In \cite{Mel19}, Melentijevi\'c obtained the following formula (\cite[p. 1051]{Mel19}):

\begin{align}\label{eqn:Melen}
C(\rho \bfe_1,\bfell) ~=~& \frac{n(n-2)}{2\pi} \frac {1} {1-\rho^2}
\int\limits_{-1}^1 \left| \frac{n-2}n \rho \cos\alpha -x \right| \\
& \quad \times \Biggl\{\int\limits_{-\sqrt{1-x^2}}^{\sqrt{1-x^2}} \frac{(1-x^2-y^2)^{\frac{n}2-2}}
{(1-2\rho x \cos \alpha - 2 \rho y \sin\alpha + \rho^2)^{\frac{n}2-1}} dy\Biggr\}dx. \notag
\end{align}
So, we start with an expansion of the inner integral in \eqref{eqn:Melen}.

\begin{lemma}\label{lem:innerintl}
For $\rho \in [0,1]$ and $\alpha \in [0,\pi] $ we have
\begin{align}
\label{eqn:innerintl2}
\int\limits_{-\sqrt{1-x^2}}^{\sqrt{1-x^2}} & \frac {(1-x^2-y^2)^{\frac{n}2-2}}
{\left[1-2\rho (x   \cos \alpha +  y \sin \alpha) + \rho^2 \right]^{\frac{n}2-1}} dy \\
&~=~ \frac {\Gamma \left(\frac {1}{2}\right) \Gamma \left(\frac {n-2}{2}\right)}{\Gamma \left(\frac {n-1}{2}\right)}
(1-x^2)^{\frac {n-3}{2}}
\sum_{k=0}^{\infty} \frac {k!}{(n-2)_k} C_k^{\frac {n-2}{2}}(x) C_k^{\frac {n-2}{2}}(\cos \alpha) \rho^k. \notag
\end{align}
\end{lemma}

\begin{proof}
By making the change of variables $y=\sqrt{1-x^2}s$, we get
\begin{align*}
\int\limits_{-\sqrt{1-x^2}}^{\sqrt{1-x^2}} & \frac {(1-x^2-y^2)^{\frac{n}2-2}}
{\left[1-2\rho (x   \cos \alpha +  y \sin \alpha) + \rho^2 \right]^{\frac{n}2-1}} dy \\
&~=~ (1-x^2)^{\frac {n-3}2} \int\limits_{-1}^{1}  \frac {(1-s^2)^{\frac{n}2 - 2} ds}
{\left[1-2\rho \left(x   \cos \alpha +  s\, \sqrt{1-x^2} \sin \alpha \right) + \rho^2 \right]^{\frac{n}2-1}}.
\end{align*}
So it suffices to prove that
\begin{align}
\label{eqn:innerintl}
\int\limits_{-1}^{1} & \frac {(1-s^2)^{\frac{n}2 - 2} ds}
{\left[1-2\rho \left(x   \cos \alpha +  s\, \sqrt{1-x^2} \sin \alpha \right) + \rho^2 \right]^{\frac{n}2-1}} \\
&~=~ \frac {\Gamma \left(\frac {1}{2}\right) \Gamma \left(\frac {n-2}{2}\right)}{\Gamma \left(\frac {n-1}{2}\right)}
\sum_{k=0}^{\infty} \frac {k!}{(n-2)_k} C_k^{\frac {n-2}{2}}(x) C_k^{\frac {n-2}{2}}(\cos \alpha)\, \rho^k. \notag
\end{align}

We divide the proof of \eqref{eqn:innerintl} into two cases, according to the dimension $n$.

\subsubsection*{Case I: $n>3$.}

By the generating relation \eqref{eqn:generatingformula}, we see that the left hand side of \eqref{eqn:innerintl}
equals
\[
\sum_{k=0}^{\infty} \left\{ \int\limits_{-1}^{1} (1-s^2)^{\frac{n}2 - 2}
C_k^{\frac{n-2}{2}} (x \cos \alpha + s\, \sqrt{1-x^2} \sin \alpha)  ds \right\} \rho^k.
\]
By the addition theorem \eqref{eqn:addthm}, with $x=\cos \theta$ and $t=\cos \psi$, we have
\begin{align*}
C_k^{\frac{n-2}{2}}(& x \cos \alpha + s\, \sqrt{1-x^2} \sin \alpha)\\
&~=~  \sum_{j=0}^{k} \beta_{k,j}\cdot \big(\sqrt{1-x^2}\big)^j (\sin\alpha)^j  C_{k-j}^{\frac{n-2}{2}+j} (x)
C_{k-j}^{\frac{n-2}{2}+j} (\cos\alpha)
C_j^{\frac {n-3}{2}}(s), \notag
\end{align*}
where
\[
\beta_{k,j} ~:=~ \frac {\Gamma(n-3)}{\Gamma^2(\frac {n-2}{2})} \frac {2^{2j} \Gamma(k-j+1)
\Gamma^2 (\frac {n-2}{2} +j)} {\Gamma (k+n-2+j)} (n-3+2j).
\]
It follows that
\begin{align*}
\textrm{LHS of } \eqref{eqn:innerintl} 
~=~& \sum_{k=0}^{\infty} \sum_{j=0}^{k} \beta_{k,j}\cdot \big(\sqrt{1-x^2}\big)^j (\sin\alpha)^j  C_{k-j}^{\frac{n-2}{2}+j} (x)
C_{k-j}^{\frac{n-2}{2}+j} (\cos\alpha) \\
& \quad \times \left\{ \int\limits_{-1}^{1} (1-s^2)^{\frac{n}2 - 2}  C_j^{\frac {n-3}{2}}(s) ds \right\} \rho^k,\\
=~& \frac {\Gamma \left(\frac {1}{2}\right) \Gamma \left(\frac {n-2}{2}\right)}{\Gamma \left(\frac {n-1}{2}\right)}
\sum_{k=0}^{\infty} \beta_{k,0}  C_{k}^{\frac{n-2}{2}} (x)
C_{k}^{\frac{n-2}{2}} (\cos\alpha) \rho^k,
\end{align*}
where in the last equality we have used the orthogonality relation \eqref{eqn:orthorelation}. Noting that $\beta_{k,0}=k!/(n-2)_k$,
this establishes the formula \eqref{eqn:innerintl} in the case $n>3$.

\subsubsection*{Case II: $n=3$.}

By making the substitute $s=\cos \psi$ in the integral and using the the generating relation \eqref{eqn:generatingformula},
we see that
\begin{align}\label{eqn:expansion2}
\textrm{LHS of } \eqref{eqn:innerintl} ~=~& \int\limits_{0}^{\pi}
\frac {d\psi} {\sqrt{1-2\rho \left(x \cos \alpha + \sqrt{1-x^2} \sin \alpha \cos \psi\right) +\rho^2}}\\
& \quad ~=~ \sum_{k=0}^{\infty} \left\{ \int\limits_{0}^{\pi}
P_k  \left(x \cos \alpha + \sqrt{1-x^2} \sin \alpha \cos \psi\right)  d\psi \right\} \rho^k, \notag
\end{align}
where $P_k(x):=C_k^{\frac {1}{2}}(x)$ is the Legendre polynomial of degree $k$.
This time we use the following addition theorem for the Legendre polynomials (see \cite[p.326-328]{WW27}):
\begin{align}\label{eqn:additionfml2}
P_k(& \cos \theta \cos \varphi + \sin \theta \sin \varphi \cos \psi) \\
& ~=~ P_k(\cos \theta) P_k(\cos \varphi)  + 2 \sum_{j=1}^{k}
\frac {(k-j)!}{(k+j)!}  P_{k}^{j} (\cos\theta) P_{k}^{j} (\cos\varphi) \cos (j\psi), \notag
\end{align}
where $P_{k}^{j}(x)$ is the associated Legendre function which is defined by
\[
P_{k}^{j}(x) ~:=~ (-1)^{j} (1-x^2)^{j/2} \frac {d^{j}}{dx^{j}} P_k(x), \qquad -1<x<1.
\]
Substituting \eqref{eqn:additionfml2} into \eqref{eqn:expansion2} yields
\begin{align*}
\textrm{LHS of } \eqref{eqn:innerintl} ~=~&
\pi \sum_{k=0}^{\infty} P_k(x) P_k(\cos \alpha) \rho^k \\
&\qquad + 2 \sum_{k=0}^{\infty} \sum_{j=1}^{k} \frac {(k-j)!}{(k+j)!} P_k^j(x) P_k^j(\cos \alpha) \int\limits_0^{\pi}
\cos (j\psi) d\psi\\
~=~& \pi \sum_{k=0}^{\infty}  C_k^{\frac {1}{2}}(x) C_k^{\frac {1}{2}}(\cos \alpha)\, \rho^k, \notag
\end{align*}
and the proof is complete.
\end{proof}

\begin{proof}[Proof of Proposition \ref{prop:keyreprn}]
Substituting \eqref{eqn:innerintl2} into \eqref{eqn:Melen} yields
\begin{align*}
C(\rho \bfe_1, \bfell)
~=~&  \frac {\Gamma \left(\frac {1}{2}\right) \Gamma \left(\frac {n-2}{2}\right)}{\Gamma \left(\frac {n-1}{2}\right)}
\frac{n(n-2)}{2\pi} \frac {1} {1-\rho^2}
\sum_{k=0}^{\infty} \frac {k!}{(n-2)_k} C_k^{\frac {n-2}{2}}(\cos \alpha) \rho^k \\
&\quad \times \int\limits_{-1}^1 \left| \frac {n-2}{n}\rho \cos\alpha -x \right| (1-x^2)^{\frac {n-3}{2}}
C_k^{\frac {n-2}{2}}(x) dx \\
=~& \frac {c_n} {1-\rho^2}
\Biggl\{\int\limits_{-1}^1 \left| \frac {n-2}{n}\rho \cos\alpha -x \right|  (1-x^2)^{\frac {n-3}{2}} dx \\
& \quad + (n-2)\rho \cos \alpha \int\limits_{-1}^1 \left| \frac {n-2}{n}\rho \cos\alpha -x \right|
(1-x^2)^{\frac {n-3}{2}} x dx \notag\\
& \quad + \sum_{k=2}^{\infty} \frac {k!}{(n-2)_k} C_k^{\frac {n-2}{2}}(\cos \alpha) \rho^k \\
&\quad \quad \times \int\limits_{-1}^1 \left| \frac {n-2}{n}\rho \cos\alpha -x \right| (1-x^2)^{\frac {n-3}{2}}
C_k^{\frac {n-2}{2}}(x) dx \Biggr\},
\end{align*}
where in the last equality we have used \eqref{eqn:specialcases}. Then, an application of Lemma \ref{lem:intbyterms},
with $\lambda=\frac {n-2}{2}$ and $s=\frac {n-2}{n}\rho \cos \alpha$, completes the proof.
\end{proof}

\section{The proof of Theorem \ref{thm:main}}

Just like in \cite{Kal17}, \cite{Mar17} and \cite{Mel19}, we shall prove the following equivalent formulation of Theorem \ref{thm:main}.

\begin{thmbis}{thm:main}\label{thm:mainprm}
For fixed $\rho \in [0,1)$, the function
$\alpha \longmapsto  C(\rho \bfe_1, \bfell)$
attains its maximum on $[0,\pi/2]$ at $\alpha=0$.
\end{thmbis}

In the sequal, we fix $0<\rho<1$ and write $\delta:=\frac {n-2}{n} \rho$. In view of Proposition \ref{prop:keyreprn}, we define
\begin{align}\label{eqn:functionH}
F(t) ~:=~&  \int\limits_{-1}^{1} |\delta t-x| (1-x^2)^{\frac {n-3}{2}} dx, \notag\\
G(t) ~:=~& (n-2)\rho t \int\limits_{-1}^{1} |\delta t-x| (1-x^2)^{\frac {n-3}{2}} x dx, \notag\\
\intertext{and}
H(t) ~:=~& \frac {2}{n^2-1}  (1-\delta^2 t^2)^{\frac {n+1}{2}} \sum_{k=2}^{\infty} \frac {(k-2)!}{(n+2)_{k-2}}
C_{k-2}^{\frac {n+2}{2}}(\delta t) C_{k}^{\frac {n-2}{2}}(t)\, \rho^{k}.
\end{align}
Hence
\begin{equation*}
C(\rho \bfe_1,\bfell) = \frac {c_n} {1-\rho^2} \left[F(\cos\alpha)+G(\cos\alpha)+H(\cos\alpha)\right].
\end{equation*}
Recall that a convex function attains its maximum over an interval at one of the end-points
and note that $C(\rho \bfe_1,\boldsymbol{\ell}_0) = C(\rho \bfe_1,\boldsymbol{\ell}_{\pi})$
(see \cite[Lemma 2.10]{Mar17}). So, we are reduced to prove the following.

\begin{thmbiss}{thm:main}\label{thm:reformlmain}
The function $F+G+H$ is convex on $[-1,1]$.
\end{thmbiss}

To this end, we first compute $F^{\prime\prime} + G^{\prime\prime} + H^{\prime\prime}$.

\begin{lemma}\label{lem:2ndder}
We have
\begin{align*}
F^{\prime\prime}&(t)  + G^{\prime\prime}(t) + H^{\prime\prime}(t) \\
&~=~ 2 \delta^2 (1-\delta^2 t^2)^{\frac {n-3}{2}} \sum_{k=0}^{\infty} \frac {k!}{(n-2)_k}
C_k^{\frac {n-2}{2}}(\delta t) C_k^{\frac {n-2}{2}}(t)\, \rho^k \\
&\qquad - \frac {4n \delta^2}{n-1} (1-\delta^2 t^2)^{\frac {n-1}{2}} \sum_{k=0}^{\infty} \frac {k!}{(n)_k}
C_k^{\frac {n}{2}}(\delta t) C_k^{\frac {n}{2}}(t)\, \rho^k \notag \\
&\qquad + \frac {2n^3 \delta^2}{(n+1)(n-1)(n-2)} (1-\delta^2 t^2)^{\frac {n+1}{2}} \sum_{k=0}^{\infty} \frac {k!}{(n+2)_k}
C_k^{\frac {n+2}{2}}(\delta t) C_k^{\frac {n+2}{2}}(t)\, \rho^k.
\end{align*}
\end{lemma}

\begin{proof}
An easy calculation gives
\[
F^{\prime\prime}(t)
~=~ 2 \delta^2 (1-\delta^2 t^2)^{\frac {n-3}{2}},
\]
or equivalently
\[
F^{\prime\prime}(t) ~=~ 2 \delta^2 (1-\delta^2 t^2)^{\frac {n-3}{2}}
C_0^{\frac {n-2}{2}}(\delta t) C_0^{\frac {n-2}{2}}(t), \label{eqn:2ndderofF}
\]
since $C_0^{\frac {n-2}{2}}(\delta t) = C_0^{\frac {n-2}{2}}(t) \equiv 1$.

Also, straightforward computations yield
\begin{align*}
G^{\prime\prime}(t)
~=~& 2 (n-2) \rho \delta^3 t^2 (1-\delta^2 t^2)^{\frac {n-3}{2}}
- \frac {4(n-2)}{n-1} \rho \delta  (1-\delta^2 t^2)^{\frac {n-1}{2}}.
\end{align*}
In view of that $C_1^{\frac {n-2}{2}}(\delta t) C_1^{\frac {n-2}{2}}(t) = (n-2)^2 \delta t^2$ and $(n-2)\rho = n\delta$,
we get
\begin{align*}
G^{\prime\prime}(t)
~=~& \frac {2}{n-2} \delta^2 (1-\delta^2 t^2)^{\frac {n-3}{2}}
C_1^{\frac {n-2}{2}}(\delta t) C_1^{\frac {n-2}{2}}(t)\, \rho \\
&\quad - \frac {4n}{n-1} \delta^2 (1-\delta^2 t^2)^{\frac {n-1}{2}} C_0^{\frac {n}{2}}(\delta t) C_0^{\frac {n}{2}}(t).
\end{align*}

Now, what is left is to show that
\begin{align}\label{eqn:2ndderofH}
H^{\prime\prime}(t)
&= 2 \delta^2 (1-\delta^2 t^2)^{\frac {n-3}{2}} \sum_{k=2}^{\infty} \frac {k!}{(n-2)_k}
C_k^{\frac {n-2}{2}}(\delta t) C_k^{\frac {n-2}{2}}(t)\, \rho^k \\
&\qquad - \frac {4n \delta^2}{n-1} (1-\delta^2 t^2)^{\frac {n-1}{2}} \sum_{k=1}^{\infty} \frac {k!}{(n)_k}
C_k^{\frac {n}{2}}(\delta t) C_k^{\frac {n}{2}}(t)\, \rho^k \notag \\
&\qquad + \frac {2n^3 \delta^2}{(n+1)(n-1)(n-2)} (1-\delta^2 t^2)^{\frac {n+1}{2}} \sum_{k=0}^{\infty} \frac {k!}{(n+2)_k}
C_k^{\frac {n+2}{2}}(\delta t) C_k^{\frac {n+2}{2}}(t)\, \rho^k. \notag
\end{align}

To this end, we differentiate \eqref{eqn:functionH} twice to obtain
\[
H^{\prime\prime}(t)  = H_1(t) + H_2(t) + H_3(t),
\]
where
\begin{align*}
H_1(t) ~:=~& \frac {2}{n^2-1}  \sum_{k=2}^{\infty} \frac {(k-2)!}{(n+2)_{k-2}}
\frac {d^2}{dt^2} \left\{(1-\delta^2 t^2)^{\frac {n+1}{2}} C_{k-2}^{\frac {n+2}{2}}(\delta t)\right\}
C_{k}^{\frac {n-2}{2}}(t)\, \rho^{k},\\
H_2(t) ~:=~& \frac {4}{n^2-1}  \sum_{k=2}^{\infty} \frac {(k-2)!}{(n+2)_{k-2}}
 \frac {d}{dt} \left\{(1-\delta^2 t^2)^{\frac {n+1}{2}} C_{k-2}^{\frac {n+2}{2}}(\delta t)\right\} \frac {d}{dt}
\left\{C_{k}^{\frac {n-2}{2}}(t)\right\}\, \rho^{k},\\
H_3(t) ~:=~& \frac {2}{n^2-1}  \sum_{k=2}^{\infty} \frac {(k-2)!}{(n+2)_{k-2}}
(1-\delta^2 t^2)^{\frac {n+1}{2}} C_{k-2}^{\frac {n+2}{2}}(\delta t) \frac {d^2}{dt^2} \left\{C_{k}^{\frac {n-2}{2}}(t)\right\}\, \rho^{k}.
\end{align*}
Repeated application of \eqref{eqn:diffofgeng2} yields that
\begin{align}\label{eqn:H1}
H_1(t) ~=~& \frac {2}{n^2-1}  \sum_{k=2}^{\infty} \frac {(k-2)!}{(n+2)_{k-2}}
\frac {(k-1)(k+n-1)}{n} \frac {k(k+n)}{n-2} \\
&\hspace{72pt} \times  \delta^2 (1-\delta^2 t^2)^{\frac {n-3}{2}} C_{k}^{\frac {n-2}{2}}(\delta t)
C_{k}^{\frac {n-2}{2}}(t)\, \rho^{k} \notag\\
=~& 2 \delta^2 (1-\delta^2 t^2)^{\frac {n-3}{2}} \sum_{k=2}^{\infty} \frac {(k)!}{(n-2)_{k}}
C_{k}^{\frac {n-2}{2}}(\delta t) C_{k}^{\frac {n-2}{2}}(t)\, \rho^{k}. \notag
\end{align}
Also, using \eqref{eqn:diffofgeng2} and \eqref{eqn:diffofgeng1} we obtain
\[
H_2(t) ~=~ - \frac {4(n-2)}{n-1}  \delta\, (1-\delta^2 t^2)^{\frac {n-1}{2}} \sum_{k=2}^{\infty} \frac {(k-1)!}{(n)_{k-1}}
C_{k-1}^{\frac {n}{2}}(\delta t) C_{k-1}^{\frac {n}{2}}(t)\, \rho^{k}.
\]
By changing the summation index from $k$ to $k+1$ and recalling that $\frac {n-2}{n}\rho = \delta$, we get
\begin{equation}\label{eqn:H2}
H_2(t) ~=~ - \frac {4n}{n-1}  \delta^2 (1-\delta^2 t^2)^{\frac {n-1}{2}} \sum_{k=1}^{\infty} \frac {k!}{(n)_{k}}
C_{k}^{\frac {n}{2}}(\delta t) C_{k}^{\frac {n}{2}}(t)\, \rho^{k}.
\end{equation}
In the similar way we obtain
\begin{align}\label{eqn:H3}
H_3(t) ~=~&  \frac {2n(n-2)}{n^2-1}  (1-\delta^2 t^2)^{\frac {n+1}{2}} \sum_{k=2}^{\infty} \frac {(k-2)!}{(n+2)_{k-2}}
C_{k-2}^{\frac {n+2}{2}}(\delta t) C_{k-2}^{\frac {n+2}{2}}(t)\, \rho^{k}\\
=~& \frac {2n^3}{(n-2)(n-1)(n+1)} \delta^2 (1-\delta^2 t^2)^{\frac {n+1}{2}} \sum_{k=0}^{\infty} \frac {k!}{(n+2)_{k}}
C_{k}^{\frac {n+2}{2}}(\delta t) C_{k}^{\frac {n+2}{2}}(t)\, \rho^{k}. \notag
\end{align}
Summing up \eqref{eqn:H1}, \eqref{eqn:H2} and \eqref{eqn:H3} leads to the desired equality \eqref{eqn:2ndderofH},
and the proof of the lemma is complete.
\end{proof}

We are now turning to the proof of Theorem \ref{thm:reformlmain}.

\begin{proof}[Proof of Theorem \ref{thm:reformlmain}]
It follows from Lemmas \ref{lem:2ndder} and \ref{lem:Gegenbauer1874} that
\begin{align*}
F^{\prime\prime}(t)  + G^{\prime\prime}(t) + H^{\prime\prime}(t)  =~& 2 \delta^2 (1-\delta^2 t^2)^{\frac {n-3}{2}} \sum_{k=0}^{\infty} \rho^k
\int\limits_{-1}^1 C_k^{\frac {n-2}{2}}(z) K_{\frac {n-2}{2}}(\delta t,t,z) dz\\
&\quad - \frac {4n \delta^2}{n-1} (1-\delta^2 t^2)^{\frac {n-1}{2}} \sum_{k=0}^{\infty} \rho^k
\int\limits_{-1}^1 C_k^{\frac {n}{2}}(z) K_{\frac {n}{2}}(\delta t,t,z) dz \\
&\quad + \frac {2n^3 \delta^2}{(n+1)(n-1)(n-2)} (1-\delta^2 t^2)^{\frac {n+1}{2}} \\
& \qquad \qquad \times \sum_{k=0}^{\infty} \rho^k
\int\limits_{-1}^1 C_k^{\frac {n+2}{2}}(z) K_{\frac {n+2}{2}}(\delta t,t,z) dz.
\end{align*}
Combining with the generating relation \eqref{eqn:generatingformula}, this yields
\begin{align}\label{eqn:finalstep}
F^{\prime\prime}&(t)  + G^{\prime\prime}(t) + H^{\prime\prime}(t) \\
&~=~ 2 \delta^2 (1-\delta^2 t^2)^{\frac {n-3}{2}}
\int\limits_{-1}^1 (1-2\rho z+ \rho^2)^{-\frac {n-2}{2}} K_{\frac {n-2}{2}}(\delta t,t,z) dz \notag\\
&\qquad - \frac {4n \delta^2}{n-1} (1-\delta^2 t^2)^{\frac {n-1}{2}}
\int\limits_{-1}^1 (1-2\rho z+ \rho^2)^{-\frac {n}{2}} K_{\frac {n}{2}}(\delta t,t,z) dz \notag \\
&\qquad + \frac {2n^3 \delta^2}{(n+1)(n-1)(n-2)} (1-\delta^2 t^2)^{\frac {n+1}{2}} \notag \\
& \hspace{72pt} \times
\int\limits_{-1}^1 (1-2\rho z+ \rho^2)^{-\frac {n+2}{2}} K_{\frac {n+2}{2}}(\delta t,t,z) dz. \notag
\end{align}
We shall show that $F^{\prime\prime}(t)  + G^{\prime\prime}(t) + H^{\prime\prime}(t) \geq 0$ for all $t\in (-1,1)$.
In view of \eqref{eqn:finalstep}, it suffices to show that the function
\begin{align*}
L(t,z) ~:=~& 2 (1-\delta^2 t^2)^{\frac {n-3}{2}}
(1-2\rho z+ \rho^2)^{-\frac {n-2}{2}} \widetilde{K}_{\frac {n-2}{2}}(\delta t,t,z) \\
&\quad - \frac {4n}{n-1} (1-\delta^2 t^2)^{\frac {n-1}{2}}
 (1-2\rho z+ \rho^2)^{-\frac {n}{2}} \widetilde{K}_{\frac {n}{2}}(\delta t,t,z) \\
&\quad  + \frac {2n^3}{(n+1)(n-1)(n-2)} (1-\delta^2 t^2)^{\frac {n+1}{2}} \\
& \hspace{64pt} \times
 (1-2\rho z+ \rho^2)^{-\frac {n+2}{2}} \widetilde{K}_{\frac {n+2}{2}}(\delta t,t,z)
\end{align*}
is nonnegative on the region
\[
\Omega := \{(t,z): -1<t,z<1 \text{ and }  1-\delta^2t^2 -t^2-z^2+2\delta t^2z>0\}.
\]
But it is easy to check that
\begin{align*}
L(t,z) ~=~&\frac {2\Gamma(\frac {n-1}{2})} {\Gamma(\frac {n-2}{2}) \Gamma(\frac {1}{2})}
\frac {(1-\delta^2 t^2 -t^2-z^2+2\delta t^2 z)^{\frac {n-4}{2}} }
{(1-2\rho z+ \rho^2)^{\frac {n+2}{2}} (1-t^2)^{\frac {n+1}{2}}} \\
& \qquad \times \biggl\{ (1-2\rho z+\rho^2)^2 (1-t^2)^2 \\
& \qquad \qquad - \frac {2n}{n-2} (1-2\rho z +\rho^2) (1-t^2) (1-\delta^2t^2 -t^2-z^2+2\delta t^2z)\\
& \qquad\qquad  + \frac {n^2}{(n-2)^2} (1-\delta^2t^2 -t^2-z^2+2\delta t^2z)^2\biggr\},
\end{align*}
which is obviously nonnegative on $\Omega$. We have thus proved the theorem.
\end{proof}
%

\end{document}